\documentclass[11pt]{amsart}

\usepackage[utf8]{inputenc}
\usepackage[english]{babel}
\usepackage[foot]{amsaddr}
\usepackage{hyperref}
\usepackage{longtable}
\usepackage{multirow}
\input xy
\xyoption{all}
\usepackage{mathrsfs}
\usepackage[dvipsnames]{xcolor}

\usepackage[a4paper]{geometry}
\usepackage{indentfirst}

\usepackage[pdftex]{graphicx}

\usepackage{amsmath}
\allowdisplaybreaks  
\usepackage{amssymb} 
\usepackage{amsfonts} 
\usepackage{enumerate} 

\numberwithin{equation}{section}

\newtheorem{theorem}{Theorem}[section]
\newtheorem{proposition}[theorem]{Proposition}
\newtheorem{definition}[theorem]{Definition}
\newtheorem{lemma}[theorem]{Lemma}
\newtheorem{corollary}[theorem]{Corollary}

\newcommand{\R}{\mathbb{R}}
\newcommand{\Q}{\mathbb{Q}}

\newcommand{\N}{\mathbb{N}}

\newcommand{\Lip}{\operatorname{Lip}}

\newcommand{\LIP}{\operatorname{LIP}}
\newcommand{\md}{\operatorname{md}}
\newcommand{\ud}{\operatorname{d}}
 \newcommand{\Leb}[1]{{\mathscr L}^{#1}}

\makeatletter
\@namedef{subjclassname@2020}{%
  \textup{2020} Mathematics Subject Classification}
\makeatother

\title{Connections between metric differentiability and rectifiability}

\author[I. Caama\~{n}o]{Iv\'an Caama\~{n}o$^{1}$}
\address{$^{1}$ Depto. de An\'alisis Ma\-te\-m\'a\-ti\-co y Matem\'atica Aplicada, UCM\\
28040 Madrid, Spain.}
\email{ivancaam@ucm.es}

\author[E. Durand-Cartagena]{Estibalitz Durand-Cartagena$^{2}$}
\address{$^{2}$ Depto. de Matem\'atica Aplicada, ETSI Industriales, UNED\\
28040 Madrid, Spain.} 
\email{edurand@ind.uned.es}

\author[J. {\'A}. Jaramillo]{Jes{\'u}s {\'A}. Jaramillo$^{1}$}
\email{jaramil@mat.ucm.es}

\author[{\'A}. Prieto]{{\'A}ngeles Prieto$^{1}$}
\email{angelin@mat.ucm.es}

\author[E. Soultanis]{Elefterios Soultanis$^{3}$}
\address{$^{3}$ Department of Mathematics and Statistics, University of Jyv\"askyl\"a\\
Seminaarinkatu 15, PO Box 35, FI-40014 University of Jyv\"askyl\"a, Finland.} 
\email{elefterios.e.soultanis@jyu.fi}

\thanks{The research of I. C., J. {\'A}. and E.D-C. is partially supported by grant PID2022-138758NB-I00 (Spain). E.S.'s research is supported by  the Finnish Academy grant no. 355122.}


\begin{document}

\maketitle
\begin{abstract}
We combine Kirchheim's metric differentials with Cheeger charts in order to establish a non-embeddability principle for any collection $\mathcal C$ of Banach (or metric) spaces: if a metric measure space $X$ bi-Lipschitz embeds in some element in $\mathcal C$, and if every Lipschitz map $X\to Y\in \mathcal C$ is differentiable, then $X$ is rectifiable. This gives a simple proof of the rectifiability of Lipschitz differentiability spaces that are bi-Lipschitz embeddable in Euclidean space, due to Kell--Mondino. Our principle also implies a converse to Kirchheim's theorem: if all Lipschitz maps from a domain space to arbitrary targets are metrically differentiable, the domain is rectifiable. We moreover establish the compatibility of metric and w$^*$-differentials of maps from metric spaces in the spirit of Ambrosio--Kirchheim.

\end{abstract}

\section{Introduction}

Going beyond geometric measure theory in Euclidean space, \emph{metric differentiability}, introduced by Kirchheim \cite{Kirchheim} has become an indispensable tool in studying rectifiability of metric spaces. A metric measure space $X=(X,d,\mu)$ is called $n$-rectifiable if $\mu\ll\mathcal H^n$ and $\mu\left( X\backslash\bigcup_{i=1}^\infty \psi_i(E_i) \right) =0$
for some countable family of Lipschitz maps $\psi_i\in \LIP(E_i,X)$ defined on $\Leb{n}$-measurable sets $E_i\subset \R^n$. 
Kirchheim \cite{Kirchheim} showed that every Lipschitz map $f:E\to X$ from a measurable set $E\subset \R^n$ into a metric space is \emph{metrically differentiable}: for  $\Leb{n}$-a.e. $x\in E$ there exists a seminorm $\md_xf$ on $\R^n$ so that
\begin{align}\label{eq:kirch-chart}
	d(f(y),f(z))=\md_xf(y-z)+o(d(y,x)+d(x,z)).
\end{align}
As an illustration of their use, metric differentials give rise to area and co-area formulae, and can be used to show that the maps $\psi_i$ in the definition of rectifiability can be taken to be bi-Lipschitz, see \cite{Kirchheim,amb-kir00}. Recently, Bate \cite{ba22} has obtained a characterization of rectifiability in terms of Gromov-Hausdorff approximations, a much weaker condition than (bi-)Lipschitz maps.

Parallel developments in analysis on metric spaces lead to the seminal work of Cheeger \cite{che99} introducing what have come to be known as Lipschitz differentiability spaces (LDS for short). These are spaces covered by countably many \emph{Cheeger charts}. A Cheeger chart is a pair $(U,\varphi)$ consisting of a Borel set $U\subset X$ with $\mu(U)>0$ and $\varphi\in \LIP(X,\R^n)$ such that every $f\in \LIP(X)$ is differentiable $\mu$-a.e. on $U$ with respect to $(U,\varphi)$: for $\mu$-a.e. $x\in U$ there exists a unique linear map $\ud_xf\in (\R^n)^*$ so that
\begin{align}\label{eq:cheeger-chart}
	f(y)-f(x)=\ud_xf(\varphi(y)-\varphi(x))+o(d(x,y)).
\end{align}
Cheeger showed that spaces endowed with a doubling measure supporting some Poincar\'{e} inequality (in short, PI-spaces) admit a countable covering by such charts, thus establishing a Rademacher-type almost everywhere differentiability result for {Lipschitz} functions from a metric space. This has lead to a rich theory of Sobolev spaces and first order calculus on PI-spaces \cite{hei01,HKST07,bjo11} independently of any rectifiability assumptions. Indeed, one of the motivations in \cite{che99} is to obtain non-embedding results for purely unrectifiable PI-spaces such as Carnot groups and Laakso spaces, see e.g \cite[Theorem 14.2]{che99}, which have since received much attention \cite{CheKlei1, CheKlei2, CheKlei3, Da}. 
The general principle is that differentiability and embeddability imply rectifiability, providing an obstruction for unrectifiable PI-spaces to admit bi-Lipschitz embeddings. The connection between PI-spaces, Cheeger differentiability and rectifiability has also been thoroughly explored in the works of \cite{sch16b,bate17, DaKlei}. For further connections to PDE's and uniform rectifiability, we refer the reader to \cite{az21,ADeRi19}.

In this note we combine Cheeger's idea of differentiability charts with Kirchheim's notion of metric differential. Throughout the paper, a \emph{chart} refers to a pair $(U,\varphi)$ consisting of a Borel set $U\subset X$ with $\mu(U)>0$ and a Lipschitz map $\varphi:X\to \R^n$. The number $n$ is called the dimension of the chart. Below, $\mathcal S(\R^n)$ denotes the set of all seminorms on $\R^n$, equipped with the metric $\delta(s,s')=\sup_{|v|\le 1}|s(v)-s'(v)|$.

\begin{definition}\label{def:metricdifferentiability}
	Given a map $f:A\subset X\to Y$ into a metric space $Y$, we say that $f$ admits a metric differential with respect to the chart $(U,\varphi)$ if there exists a Borel map $\md f:A\cap U\to \mathcal S(\R^n)$ satisfying, for $\mu$-a.e. $x\in A\cap U$,
	\begin{equation}\label{eq:kir-chart-diff}
		\limsup_{A\ni y\to x}\frac{|d_Y(f(y),f(x))-\md_xf(\varphi(y)-\varphi(x))|}{d(x,y)}=0. 
	\end{equation}
\end{definition}

See also \cite{CheKleiSc} for an alternative approach to metric differentiation of mappings between metric spaces which relies on metric differentiation along curves, and \cite{GT} for an extension of Kirchheim's result to maps defined on strongly rectifiable metric spaces. The definition above covers maps from a subset $A\subset X$. The case $A=U$ is the metric analogue of \emph{weak} Cheeger charts while, if $\mu(A\cap U)=0$, the condition is vacuous. We also remark that our definition (see also \cite{amb-kir00}) is slightly weaker than Kirchheim's original definition, where $d(f(y),f(z)))=\md_xf(z-y)+o(d(x,y)+d(x,z))$.
 
Charts with respect to which every $f\in \LIP(U)$ admits a metric differential are weak Cheeger charts, see Proposition \ref{prop:kir-implies-che}. Moreover, metric differentials are compatible with $w^*$-differentials, whenever both exist, see Section \ref{sec:weak-star}. Notice that we do not impose uniqueness of the metric differential in Definition \ref{def:metricdifferentiability}. Indeed, uniqueness among seminorms is a much stronger requirement than uniqueness among linear maps, implying in particular the density of directions realised by $\varphi$. In Theorem \ref{thm:main-thm} we will instead impose the a priori weaker condition
\begin{align}\label{eq:lin-uniqueness}
	\Lip(v\cdot\varphi|_U)>0\ \mu\textrm{-a.e. on $U$ for any }v\in S^{n-1},
\end{align}
which is equivalent to uniqueness of linear differentials, see \cite[Lemma 2.1]{BaSp}. We remark that if $(U,\varphi)$ is a weak Cheeger chart, then the density of directions holds (see \cite[Lemma 9.1]{Ba}) and thus, a posteriori, metric differentials are unique, see Section \ref{sec:KDS-implies-LDS}.


Our main result is a rectifiability criterion which relates metric differentiability of maps into a given target class to bi-Lipschitz embeddability in the same class. Our proof gives a conceptually simple argument covering several non-embeddability results known in the literature, see Corollary \ref{cor:non-embed}. 

\begin{theorem}\label{thm:main-thm}
	
Let $(U,\varphi)$ be an $n$-dimensional chart in $X$ satisfying \eqref{eq:lin-uniqueness}, and $\mathcal C$ a collection of metric spaces so that some $Y\in \mathcal C$ contains a non-trivial geodesic. If every Lipschitz map $U\to Y\in\mathcal C$ admits a metric differential with respect to $(U,\varphi)$, and $U$ bi-Lipschitz embeds into some $Y\in \mathcal C$, then $(U,d,\mu|_U)$ is $n$-rectifiable.

More precisely, $\mu|_U\simeq \mathcal H^n|_U$ and there are disjoint Borel sets $U_i$ with $\mu\big(U\setminus \bigcup_i U_i)=0$ so that $\varphi|_{U_i}$ is bi-Lipschitz for each $i$.

\end{theorem}

We make a few remarks: 

1) A non-trivial geodesic in a metric space $Z$ is an isometric embedding $\gamma:[a,b]\to Z$ for some $a<b$. The existence of one in some space $Z\in \mathcal C$ guarantees (under the hypotheses of Theorem \ref{thm:main-thm}) that all real valued Lipschitz functions admit a metric differential (see Lemma \ref{lem:geod-implies}), a fact which self-improves to the existence of \emph{linear} differentials (Proposition \ref{prop:kir-implies-che}). We give an elementary argument in Section \ref{sec:KDS-implies-LDS}, but this can also be shown along the lines of \cite{Ba}.

2) We could alternatively require that every Lipschitz map $X\to Y\in\mathcal C$ admits a metric differential and that there exists $f\in \LIP(X,Y)$ such that $f|_U$ is bi-Lipschitz for some $Y\in \mathcal C$. In particular, if we assume that $(X,Y)$ has the Lipschitz extension property for each $Y\in\mathcal C$, the claim of the theorem is true assuming metric differentiability of every Lipschitz map $X\to Y\in\mathcal C$.


3) A noteworthy consequence of the fact that $(U,\varphi)$ is a weak Cheeger chart (Proposition \ref{prop:kir-implies-che}) is that $\varphi_\#(\mu|_U)\ll\Leb{n}$. This is by now well-known (see \cite{DeMaRi} for a proof) and follows from the deep result of De Philippis--Rindler \cite{DeRi}. Together with the bi-Lipschitz decomposition, this implies the mutual absolute continuity of $\mu|_U$ and $\mathcal H^n|_U$, and completes the proof of rectifiability of $(U,d,\mu|_U)$.

Now, we list some straightforward consequences of Theorem \ref{thm:main-thm}. Notice that part (a) of the following corollary provides a simpler proof of \cite[Theorem 3.7]{KeMo}.

\begin{corollary}\label{cor:non-embed} $\,$
	\begin{itemize}
		\item[(a)] If $X$ is an LDS admitting a bi-Lipschitz embedding into a Euclidean space, then $(X,d,\mu)$ is rectifiable. 
		\item[(b)] More generally, let $V$ be a Banach space. If every Lipschitz map $X\to V$ is metrically differentiable, and $X$ bi-Lipschitz embeds into $V$, then $X$ is rectifiable.
		\item[(c)] If every Lipschitz map $f:X\to c_0$ is metrically differentiable, then $X$ is rectifiable. 
		\item[(d)] If every Lipschitz map from $X$ to an arbitrary target is metrically differentiable, then $X$ is rectifiable.
	\end{itemize}
\end{corollary}

In particular, since non-Abelian Carnot groups and the Laakso space are purely unrectifiable RNP-Lipschitz differentiability spaces (see \cite{amb-kir00,Laakso,CheKlei2,Da}), they do not admit a bi-Lipschitz embedding into an RNP-Banach space. Note that (a), (b) and (d) are immediate consequences of Theorem \ref{thm:main-thm} with $\mathcal C=\{\R^n\},\{V\}, \{X\times\R\}$, 
respectively, and (c) follows readily from Theorem \ref{thm:main-thm} and the fact that every separable metric space bi-Lipschitz embeds into $c_0$ \cite[Theorem 3.12]{Hei03}. A similar conclusion holds with $c_0$ replaced by $\ell^\infty$ or $C[0,1]$, since both contain isometric copies of every separable metric space. Since all three of these are non-RNP Banach spaces, not all Lipschitz maps into them can be (linearly) differentiable.


%

\section{Preliminaries}
Throughout this paper the triplet $(X,d,\mu )$ denotes a complete separable metric space endowed with a measure $\mu$ which is Borel regular and finite on bounded sets. In particular, $\mu$ is Radon (see \cite[Corollary 3.3.47]{HKST07}).


For a mapping $f:(X,d_X)\rightarrow (Y,d_Y)$ between metric spaces we define the pointwise Lipschitz constant as
$$\mathrm{Lip}\, f(x):=\limsup_{y\to x}\frac{d_Y(f(x),f(y))}{d_X(x,y)}.$$
We denote by $\mathrm{LIP}(X,Y)$ the space of Lipschitz mappings and, for the particular case of $Y=\R$, we use the notation $\mathrm{LIP}(X)$.

In order to give a formal definition of rectifiability, we briefly recall the notion of {\em Hausdorff measure}. For $s> 0$, first fix $\delta >0$ and consider, for any set $E\subset X$,
$$\mathcal{H}_\delta^s(E):=\inf\left\{ \sum_{i=1}^\infty \mathrm{diam}(E_i)^s \right\},$$
where the infimum is taken over all countable covers of $E$ by sets $E_i\subset X$ with $\mathrm{diam}(E_i)<\delta$ for all $i\in\N$. Then the $s$-\textit{dimensional Hausdorff measure} is defined as
$$\mathcal{H}^s(E):=\lim_{\delta\to 0}\mathcal{H}_\delta^s(E).$$

Given two measures $\mu$ and $\nu$, we say that $\mu$ is {\em absolutely continuous} with respect to $\nu$ (denoted $\mu\ll\nu$) if whenever $A\subset X$ such that $\nu (A)=0$ then $\mu (A)=0$.

\begin{definition}{\em ($n$-rectifiable space)}\label{def:n-rect}
We say that  $(X,d,\mu )$ is {\em $n$-rectifiable} if $\mu\ll\mathcal H^n$ and there exists a countable collection of Lipschitz maps $\psi_i:E_i \rightarrow X$ from $\mathscr L^n$-measurable sets $E_i \subset \mathbb R^n$ with $\mu\left( X\backslash\bigcup_{i=1}^\infty \psi_i(E_i) \right) =0$.
\end{definition}

An {\em $n$-dimensional chart} on $(X,d,\mu)$ is a pair $(U,\varphi )$ such that $U\subset X$ is Borel and $\varphi :X\rightarrow \R^n$ is Lipschitz.

\begin{definition}{\em (weak Cheeger chart)}
We say that an $n$-dimensional chart $(U,\varphi )$ on $(X,d,\mu)$ is a \emph{weak Cheeger chart} if for every mapping $f\in \mathrm{LIP}(X)$ and $\mu-$almost every $x\in U$ there exists a unique linear map $\ud_xf\in (\R^n)^*$ such that
\begin{equation}\label{eq:weakcheegerchart}
		\limsup_{U\ni y\to x}\frac{|f(y)-f(x)-\ud_xf(\varphi(y)-\varphi(x))|}{d(x,y)}=0. 
	\end{equation}
\end{definition}
As mentioned in the introduction, if \eqref{eq:weakcheegerchart} holds without the restriction $y\in U$ then $(U,\varphi )$ is called a {\em Cheeger chart}. If $X$ can be decomposed, up to a $\mu$-null set, into a countable union of (weak) Cheeger charts then $X$ is called a (weak) {\em Lipschitz differentiability space}, or (weak) LDS in short. It is interesting to notice that, if porous sets in $X$ have zero measure, a weak Cheeger chart is automatically a Cheeger chart (see \cite[Remark 2.11]{Ke2} and \cite[Proposition 2.8]{BaSp}). Also, a countable union of LDS's might not be an LDS (see \cite{bate17} for an example), whereas countable unions of weak LDS are trivially a weak LDS.


In order to study mappings into a metric space where linearity is absent, a generalized definition of differentiability is necessary, and leads to the concept of {\em metric differential} introduced by Kirchheim in \cite{Kirchheim} (see Definition \ref{def:metricdifferentiability}). In this scenario, the metric differential $\md_xf$ of a mapping $f:X\rightarrow Y$ cannot be a linear map, as defined in a (weak) Cheeger chart. Instead, it is substituted by a {\em seminorm} in $\R^n$, that is, a subadditive non-negative function. More precisely, the distances $d_Y(f(y),f(x))$ behave like a seminorm (equation \eqref{eq:kir-chart-diff}) instead of $f(y)-f(x)$ behaving linearly (equation \eqref{eq:cheeger-chart}), in both cases up to a first order error.

%
%
%

In Section \ref{sec:weak-star} we will also consider the notion of {\em $w^*$-differential}, a weaker version of differentiability for maps taking values in duals to separable Banach spaces. Namely, if $V$ is a separable Banach space, we can consider the {\em $w^*$-topology} in $V^*$, that is,
$$w^*-\lim_{j\to \infty}w_j=w\Longleftrightarrow \lim_{j\to \infty}\langle w_j,v\rangle =\langle w,v\rangle \quad\forall v\in V,$$
whenever $w_j,w\in V^*$. Here $\langle\cdot ,\cdot\rangle$ denotes the standard duality $\langle v,w\rangle = w(v)$ for $v\in V$ and $ w \in V^*$.

\begin{definition}{\em ($w^*$-differentiability)}\label{def:weakdifferentiability}
	Let $(U,\varphi)$ be a 
	chart and $V$ a separable Banach space. Given a map $f:U\to V^*$, we say that $f$ is {\em $w^*$-differentiable} with respect to the chart $(U,\varphi)$ if for $\mu$-a.e. $x\in U$ there exists a unique linear map $D_xf:\R^n\to V^*$ such that
	
	\begin{align}\label{eq:w-star-differential}
		\limsup_{U\ni y\to x}\frac{|\langle v,f(y)-f(x)-D_xf(\varphi(y)-\varphi(x))\rangle|}{d(y,x)} =0\quad\textrm{for all }v\in V.
	\end{align}
\end{definition}

\section{Metric and linear differentials}\label{sec:KDS-implies-LDS}

Let $(U,\varphi)$ be an $n$-dimensional chart satisfying \eqref{eq:lin-uniqueness}. Note that for every $x\in U$ and any sequence  $(x_j) \subset U\setminus\{x\}$  with $x_j \to x$, the Lipschitz property of $\varphi$ gives that $\left(\displaystyle \frac{\varphi(x_j)-\varphi(x)}{d(x_j,x)}\right)$ has a convergent subsequence.
We denote by $L(\varphi,x)\subset \R^n$ the set of limit points of sequences $\left(\displaystyle \frac{\varphi(x_j)-\varphi(x)}{d(x_j,x)}\right)$ with $(x_j)$ as above. 

From \eqref{eq:lin-uniqueness} it follows that $L(\varphi,x)$ spans $\R^n$ for $\mu$-a.e. $x\in U$ (since $\Lip(v\cdot\varphi|_U)(x)=0$ for any $v\perp
 L(\varphi,x)$). It moreover follows from \eqref{eq:lin-uniqueness} that a function can have at most one (linear) differential
 with respect to $(U,\varphi)$, see \cite[Lemma 3.3]{Ba} and \cite[Lemma 2.1]{BaSp}. If on the other hand $f\in \LIP(X)$
 admits a metric differential with respect to $(U,\varphi)$ in the sense of Definition \ref{def:metricdifferentiability} it follows
 that, for $\mu$-a.e. $x\in U$, the metric differential $\md_xf$ is uniquely determined on $L(\varphi,x)$. In
 fact, if $\displaystyle w=\lim_{j\to\infty} \displaystyle \frac{\varphi(x_j)-\varphi(x)}{d(x_j,x)} \in L(\varphi,x),$ then $$\md_xf(w)=\lim_{j\to\infty}
 \dfrac{\md_xf(\varphi(x_j)-\varphi(x))}{d(x_j,x)}=\lim_{j\to\infty} \dfrac{|f(x_j)-f(x)|}{d(x_j,x)}\, .$$
 However if $\R L(\varphi,x)$ is not dense in $\R^n$, there may exist many seminorms $s$ on $\R^n$ with $s|_{L(\varphi,x)}=\md_xf|_{L(\varphi,x)}$. This is in contrast with linear maps, which are uniquely determined by their values on a spanning set. The density of $\R L(\varphi,x)$ holds for weak Cheeger charts by \cite[Lemma 9.1]{Ba} whose proof uses Alberti representations.

Below we show that metric differentiability of \emph{every} Lipschitz function self-improves to the existence of linear differentials.


\begin{proposition}\label{prop:kir-implies-che}
	Suppose $(U,\varphi)$ is an $n$-dimensional chart satisfying \eqref{eq:lin-uniqueness}. If every $f\in \LIP(U)$ admits a metric differential with respect to $(U,\varphi)$, then $(U,\varphi)$ is a weak Cheeger chart.
\end{proposition}

Combining this with \cite[Lemma 9.1]{Ba} we have the following immediate corollary.
\begin{corollary}\label{cor:dense-directions}
Under the hypotheses of Proposition \ref{prop:kir-implies-che}, $\R L(\varphi,x)$ is dense in $\R^n$ for $\mu$-a.e. $x\in U$.
\end{corollary}

\begin{proof}[Proof of Proposition \ref{prop:kir-implies-che}]
Let $f\in\LIP(X)$. By assumption there exists a $\mu$-null set $N\subset U$ such that $L(\varphi,x)$ spans $\R^n$ and there are seminorms $\md_x(f+w\cdot \varphi)$ satisfying
\[
\limsup_{U\ni y\to x}\frac{||f(y)-f(x)+w\cdot(\varphi(y)-\varphi(x))|-\md_x(f+w\cdot \varphi)(\varphi(y)-\varphi(x))|}{d(y,x)}=0
\]
for every $x\in U\setminus N$ and $w\in \Q^n$. We fix $x\in U\setminus N$.

If $v=\lim_{j\to\infty}\frac{\varphi(x_j)-\varphi(x)}{d(x_j,x)}\in L(\varphi,x)$, and  $a(f,x,v)=\lim_{j\to\infty}\frac{f(x_j)-f(x)}{d(x_j,x)}$ exists, then  $$|a(f,x,v)+w\cdot v|=\md_x(f+w\cdot \varphi)(v)$$ for all $w\in \Q^n$. If we consider now a different sequence $x'_j \to x$ such that $v=\lim_{j\to\infty}\frac{\varphi(x'_j)-\varphi(x)}{d(x'_j,x)}$ and $a'(f,x,v)=\lim_{j\to\infty}\frac{f(x'_j)-f(x)}{d(x'_j,x)}$ exists, then 
$$|a'(f,x,v)+w\cdot v|=|a(f,x,v)+w\cdot v|$$
for all $w\in \Q^n$ and thus we conclude that $a'(f,x,v)=a(f,x,v)$. It follows that the map $L_f:L(\varphi,x)\to\R$ given by
\begin{align*}
L_f(v):=\lim_{j\to\infty}\frac{f(x_j)-f(x)}{d(x_j,x)},\textrm{ whenever }v=\lim_{j\to\infty}\frac{\varphi(x_j)-\varphi(x)}{d(x_j,x)}
\end{align*}
is well-defined 
and satisfies
\begin{align}\label{eq:w-identity}
	\md_x(f+w\cdot\varphi)(v)=|L_f(v)+w\cdot v|\textrm{ for all } w\in \Q^n.
\end{align}
We prove that
\begin{itemize}
	\item[(a)] $L_f(t v)=t L_f(v)$ if $v, t v\in L(\varphi,x)$, and
	\item[(b)] $L_f(v+v')=L_f(v)+L_f(v')$ if $v,v',v+v'\in L(\varphi)$.
\end{itemize}
If $v,v'\in L(\varphi,x)$ satisfy $v'=t v$, $t \in\R$, then $$|L_f(v')+w\cdot v'|=\md_x(f+w\cdot\varphi)(v')=|t|\md_x(f+w\cdot\varphi)(v)=|t||L_f(v)+w\cdot v|.$$ Thus $|L_f(v')+t w\cdot v|=|t L_f(v)+t w\cdot v|$ for all $w\in \Q^n$, implying (a). Next suppose $v,v',v+v'\in L(\varphi,x)$. From \eqref{eq:w-identity} and the fact that $\md_x(f+w\cdot\varphi)$ is a seminorm we obtain that
\begin{align*}
|L_f(v+v')+w\cdot(v+v')|\le |L_f(v)+w\cdot v|+|L_f(v')+w\cdot v'|,\quad w\in \mathbb{Q}^n.
\end{align*}
If $v'=-v$, then $L_f(v+v')=0=L_f(v)+L_f(v')$ by (a). Otherwise, there exists $w\in \Q^n$ such that $w\cdot z>0$ for $z=v,v',v+v'$. By multiplying $w$ by a sufficiently large positive number we find $w^+$ such  that $L_f(z)+w^+\cdot z>0$, $z=v,v',v+v'$. From the inequality above we obtain $L_f(v+v')\le L_f(v)+L_f(v').$ Similarly, by multiplying $w$ by a suitably large (in absolute value) negative number we find $w^-$ such that $L_f(z)+w^-\cdot z<0$, $z=v,v',v+v'$, yielding $-L_f(v+v')\le -L_f(v)-L_f(v')$. These two inequalities together prove (b).

Since $L(\varphi,x)$ spans $\R^n$ and $L_f$ satisfies (a) and (b), there exist a linear map $L:\R^n\to \R$ such that $L|_{L(\varphi,x)}=L_f$. It follows that $L$ is the Cheeger differential of $f$ at x with respect to $(U,\varphi)$. Indeed, otherwise there would exist $\varepsilon_0>0$ and a sequence $U\ni x_j\to x$ with $v:=\lim_{j\to\infty}\frac{\varphi(x_j)-\varphi(x)}{d(x_j,x)}$ such that
\begin{align*}
	\varepsilon_0\le \frac{|f(x_j)-f(x)-L(\varphi(x_j)-\varphi(x))|}{d(x_j,x)}\stackrel{j\to\infty}{\longrightarrow} |L_f(v)-L(v)|,
\end{align*}
contradicting the fact that $L_f(v)=L(v)$.
\end{proof}

\section{Metric differential and rectifiability}\label{metricrectif}



Throughout this section, we fix an $n$-dimensional chart $(U,\varphi)$ satisfying \eqref{eq:lin-uniqueness}.
\begin{lemma}\label{lem:geod-implies}
	Suppose $\mathcal C$ is a collection of metric spaces containing a space with a non-trivial geodesic, and that every Lipschitz map $U\to Y\in\mathcal C$ admits a metric differential with respect to $(U,\varphi)$. Then every $f\in\LIP(U)$ admits a metric differential with respect to $(U,\varphi)$.
\end{lemma}
\begin{proof}
	Let $\gamma:[a,b]\to Y\in\mathcal C$ be a non-trivial geodesic, and let $h:\R\to (a,b)$ be a Lipschitz diffeomorphism. By assumption, for any $f\in \LIP(U)$ the map $\tilde f=\gamma\circ h\circ f$ admits a metric differential. Since $d_Y(\tilde f(y),\tilde f(x))=|h(f(y))-h(f(x))|$ for each $x,y\in U$, it follows that $h\circ f\in \LIP(U)$ admits a metric differential. However if $\md(h\circ f)$ denotes the metric differential of $h\circ f$, we have for $\mu$-a.e. $x\in U$
\begin{align*}
	|f(y)-f(x)|&=|h^{-1}(h\circ f(y))-h^{-1}(h\circ f(x))|\\
	&=|(h^{-1})'(h\circ f(x))(h\circ f(y)-  h\circ f(x))| +o(|h\circ f(y)-h\circ f(x)|)\\
	&=|(h^{-1})'(h\circ f(x))|\md_x(h\circ f)(\varphi(y)-\varphi(x))+o(d(y,x))
\end{align*}
which implies that $\md f:=|(h^{-1})'\circ h\circ f|\md(h\circ f)$ is a metric differential of $f$. This proves the claim.
\end{proof}


We now give the proof of the main result.

\begin{proof}[Proof of Theorem \ref{thm:main-thm}] By hypothesis, there exists a bi-Lipschitz embedding $f:U\to Y$ of $U$ into some $(Y,d_Y)\in \mathcal C$, which admits a metric differential $\md f$ with respect to $(U,\varphi)$. Denote the bi-Lipschitz constant of $f$ by $L$. In particular by \eqref{eq:kir-chart-diff} we have
	\begin{equation}\label{mdlimit}
		\lim_{U\ni y\to x}\frac{\vert d_Y(f(x),f(y))-\mathrm{md}_xf(\varphi (y)-\varphi (x))\vert }{d(x,y)}=0
	\end{equation}
	for $\mu-$almost every $x\in U$. Let $N\subset U$ be a null set such that \eqref{mdlimit} holds for all $x\in U\backslash N$ and rewrite the limit in \eqref{mdlimit} as follows:
	$$\lim_{j\to \infty}F_j(x)=0\quad\text{where}\quad F_j(x):=\sup_{y\in B(x,\frac{1}{j})\cap U}\frac{\vert d_Y(f(x),f(y))-\mathrm{md}_xf(\varphi (x)-\varphi (y))\vert }{d(x,y)}.$$
	By Egorov's theorem for every $\varepsilon>0$ there exists a set $K\subset U\setminus N$ with $\mu(U\setminus K)<\varepsilon$ so that $\md_xf$ exists for every $x\in K$, and $F_j\to 0$ uniformly on $K$. Since $\mu$ is Radon we may further assume that $K$ is compact. Let $j_0\in\N$ be such that
	\begin{equation*}\label{uniformlimit2}
		\sup_{y\in B(x,\frac{1}{j_0})\cap K}\frac{\vert d_Y(f(x),f(y))-\mathrm{md}_xf(\varphi (x)-\varphi (y))\vert }{d(x,y)}\leq F_{j_0}(x)\leq \frac{1}{2L},\quad x\in K.
	\end{equation*}
	In particular, for any $x,y\in K$ with $d(x,y)<1/j_0$ we have
	\[
	\frac 1Ld(x,y)-\mathrm{md}_xf(\varphi (x)-\varphi (y))\leq d_Y(f(y),f(x))-\mathrm{md}_xf(\varphi (x)-\varphi (y))\le \frac{1}{2L}d(x,y)
	\]
	from which we obtain
	\[
	d(x,y)\le 2L\md_xf(\varphi(y)-\varphi(x))\le 2LC|\varphi (x)-\varphi (y)|,
	\]
	where $C:=\sup_{|v|\le 1}\md_xf(v)$.
	Consider a covering of $K$ by balls $\{B(x,\frac{1}{2j_0})\}_{x\in K}$. By compactness, there exist $x_1,x_2,\cdots,x_N$ such that $K\subset \bigcup_{i=1}^N B(x_i,\frac{1}{2j_0}).$
	Choose $x_0\in\{x_1,x_2,\cdots,x_N\}$ such that $\mu(K\cap B(x_0,\frac{1}{2j_0}))>0$, and
	define  $$
	A_k:=\{x\in K:\frac{1}{k}\leq \sup_{|v|\le 1}\md_xf(v)\leq k\}.
	$$
	
	Recall that $f$ is the restriction of a bi-Lipschitz mapping on $K$, so that $\sup_{|v|\le 1}\md_xf(v)\neq 0$ for all $x\in K$, implying $\bigcup_{k=1}^{\infty}A_k = K$. Thus there exists $k\geq 1$ such that $\mu(A_k\cap K\cap B(x,\frac{1}{2j_0}))>0$. Now, let $x,y\in  A_k \cap  K\cap B(x_0,\frac{1}{2j_0})$. Since $y\in A_k \cap K\cap B(x,\frac{1}{j_0})$ we have that
	\[
	d(x,y)\leq 2 Lk |\varphi (x)-\varphi (y)|,
	\]
	that is, $\varphi$ is injective on $A_k \cap K\cap B(x_0,\frac{1}{2j_0})$ and $\varphi^{-1}$ is $2Lk$-Lipschitz on $\varphi(A_k\cap K\cap B(x_0,\frac{1}{2j_0}))$. In particular, $\varphi$ is bi-Lipschitz on $A_k\cap K\cap B(x_0,\frac{1}{2j_0})$. By \cite[Proposition 3.1.1]{Kei}, there exists a countable decomposition
	$$
	U=Z \cup \bigcup_i V_i,
	$$
	where $\mu(Z)=0$ and $\left\{V_i\right\}$ is a collection of mutually disjoint measurable sets such that $\varphi|_{V_i}$ is bi-Lipschitz for each $i$.
	
	To finish the proof, we note that $(U,\varphi)$ is a weak Cheeger chart of dimension $n$ by Lemma \ref{lem:geod-implies} and Proposition \ref{prop:kir-implies-che}. In particular $\varphi_\#(\mu|_U)\ll \Leb{n}$ \cite[Theorem 1.1]{DeMaRi}. Writing $\varphi_\#(\mu|_{V_i})=\rho_i\Leb{n}$ for each $V_i$ as above, we obtain that
	\[
	\mu|_{V_i}=(\rho_i\circ\varphi^{-1})\varphi^{-1}_\#(\Leb{n}|_{\varphi(V_i)})\simeq \mathcal H^n|_{V_i}.
	\]
	Consequently $\mu|_U=\sum_i\mu|_{V_i}\simeq \sum_i\mathcal H^n|_{V_i}=\mathcal H^n|_U$.
\end{proof}


Next we prove a ``converse'' \footnote{Assuming porous sets in $X$ have measure zero, if X is rectifiable, it can be decomposed into a countable union of charts with respect to which every $f:X\to Y$ admits a metric differential.} of Corollary \ref{cor:non-embed} (d), namely that Lipschitz maps from a rectifiable space to arbitrary targets admit metric differentials. This result can be considered folklore, but we record the statement and its proof below for the reader's convenience.


\begin{proposition}\label{prop:weakKirchheim}
	Suppose $(U,\varphi)$ be a $n$-dimensional chart in a metric measure space $(X,d,\mu)$ such that $\varphi|_U$ is bi-Lipschitz and $\mu|_U\ll\mathcal H^n$ 
	Then, for any metric space $Y$, every Lipschitz map $f\in \LIP(U,Y)$ admits a metric differential with respect to $(U,\varphi)$.
\end{proposition}

\begin{proof}
	Let $g=f\circ \varphi^{-1} :\varphi(U)\rightarrow Y$. Notice that $g$ is a composition of Lipschitz mappings, so it is also Lipschitz. By Kirchheim's Rademacher Theorem \cite[Theorem 2]{Kirchheim}, for $\mathcal{H}^n$-almost every $z\in \varphi(U)$ there exists a unique seminorm $ \mathrm{md}_zg$ on $\R^{n}$ such that
	\begin{equation}\label{medieq}
		\lim_{\stackrel{y\to z}{y\in \varphi(U) }}\frac{|d_Y(g(z),g(y))-\mathrm{md}_zg(y-z)|}{|y-z|}=0.
	\end{equation}
	On the other hand, $g(\varphi (x))=f(x)$ for each $x\in
	U$. Fix $x_0\in U$ such that for $z_0:= \varphi(x_0)$ there exists a unique seminorm $ \mathrm{md}_{z_0}g$ on $\R^{n}$ such that \eqref{medieq} holds. As $\varphi$ is continuous, if $x\in U$ and $x\to x_0$, then $\varphi(x)\to z_0$. Therefore
	\begin{eqnarray*}
		&&\lim_{\underset{ x\in U }{x\to x_0 }}\frac{|d_Y(f(x_0),f(x))-\mathrm{md}_zg(\varphi(x)-\varphi(x_0))|}{d(x,x_0)} \\
		&\leq &\lim_{\underset{z=\varphi(x)}{z\to z_0}}\frac{\vert d_Y(g(z_0),g(z))-\mathrm{md}_{z_0}g(z-z_0)\vert}{\frac{1}{C}|z-z_0|}=0.
	\end{eqnarray*}
	where  $C$ is the Lipschitz constant of $\varphi$. Then $\mathrm{md}_{x_0}f:=\mathrm{md}_{\varphi(x_0)}g$ is the metric differential of $f$ at  $x_0\in U$. We finish the proof by noticing that, because $\varphi_\# (\mu_{| U}) \ll \mathcal H^{n}_{|\varphi (U)}
	$ and
	\[
	\mathcal{H}^{n}(\{z_0\in \varphi(U): \mathrm{md}_{z_0}g \text{ does not exist}\})=0,
	\]
	we conclude that
	\[
	\mu(\{x_0\in U: \mathrm{md}_{x_0}f \text{ does not exist}\})=0.
	\]
	
\end{proof}

\section{Metric and $w^*$-differentials}\label{sec:weak-star}

Differentiability of real valued Lipschitz functions gives rise to $w^*$-differentials of Lipschitz maps into the dual of a separable Banach spaces. This insight was made explicit\footnote{The authors state in \cite{amb-kir00} that $w^*$-differentiability of Lipschitz maps from $\R^n$ is a folklore result.} in \cite{amb-kir00}, where it was shown that the $w^*$-differential is compatible with the metric differential of Kirchheim, see \cite[Theorem 3.5]{amb-kir00}. In Propositions \ref{prop:w-star-differential} and \ref{prop:w-star-vs-metr-diff} below we establish the compatibility of metric and $w^*$-differentials in the setting of (weak) Cheeger charts.

\begin{proposition}\label{prop:w-star-differential}
	Let $(U,\varphi)$ be a weak Cheeger chart and $V$ a separable Banach space. Given $f\in \LIP(U,V^*)$, $f$ admits a $w^*$-differential $D_xf$ with respect to $(U,\varphi)$ for $\mu$-a.e. $x\in U$.
	
\end{proposition}

\begin{proof}
	Let $D\subset V$ be a countable dense vector space over $\Q$, and $N\subset U$ a $\mu$-null set such that the unique differential $L_x(v):=\ud_x\langle v,f\rangle\in (\R^n)^*$ of $\langle v,f\rangle$ with respect to $(U,\varphi)$ exists for every $v\in D$ whenever $x\in U\setminus N$. We fix $x\in U\setminus N$. Since
	\begin{align*}
		\langle v+w,f(y)-f(x)\rangle &=\langle v,f(y)-f(x)\rangle+\langle w,f(y)-f(x)\rangle\\
		&=(L_x(v)+L_x(w))(\varphi(y)-\varphi(x))+o(d(x,y)),
	\end{align*}
	it follows by the uniqueness of the differential that $L_x(v+w)=L_x(v)+L_x(w)$ for $v,w\in D$. Similarly $L_x(av)=aL_x(v)$. These identities together with the estimate
	\begin{align*}
		\Lip(L_x(v)\circ\varphi|_U)(x)=\Lip(\langle v,f\rangle)(x)\le \|v\|\Lip f(x)
	\end{align*}
	show that $v\mapsto L_x(v)$ is a bounded linear map $D\to ((\R^n)^*,|\cdot|_x^*)$ and thus extends to a bounded linear map $L_x:V\to (\R^n)^*$. Here $|\lambda|_x^*=\Lip(\lambda\circ\varphi|_U)(x)$ is a norm on $(\R^n)^*$, in light of the fact that a weak Cheeger chart satisfies
\eqref{eq:lin-uniqueness}.
	
	We denote by $D_xf:(\R^n,|\cdot|_x)\to V^*$ the adjoint operator ($|\cdot|_x$ the dual norm of $|\cdot|_x^*$) and note that it satisfies $\langle D_xf(z),v\rangle =\langle L_x(v),z\rangle$ for all $z\in \R^n$ and $v\in V$. To see that $D_xf$ is the $w^*$-differential of $f$, observe that if $v\in V$, we have
	\begin{align*}
		\Lip(\langle v,f-D_xf\circ\varphi|_U\rangle)(x)&\le \Lip(\langle v_i,f-D_xf\circ\varphi|_U\rangle)(x)+\|v_i-v\|_V(\Lip f(x)+\Lip (D_xf\circ\varphi|_U)(x))\\
		&\le \Lip(\langle v_i,f\rangle-L_x(v_i)\circ\varphi|_U)(x)+\|v_i-v\|_V(\Lip f(x)+\Lip (D_xf\circ\varphi|_U)(x))\\
		&=0+\|v_i-v\|_V(\Lip f(x)+\Lip (D_xf\circ\varphi|_U)(x))
	\end{align*}
	for any $v_i\in D$. Taking $v_i\to v$ we obtain \eqref{eq:w-star-differential}.
\end{proof}

\begin{proposition}\label{prop:w-star-vs-metr-diff}
	Suppose $(U,\varphi)$ is a weak Cheeger chart and $f\in \LIP(U,V^*)$, where $V$ is a separable Banach space. If $f$ admits a metric differential $\md f$ with respect to $(U,\varphi)$, then for $\mu$-a.e. $x\in U$ we have $\md_xf(z)=\|D_xf(z)\|_{V^*}$ for all $z\in \R^n$.
\end{proposition}
The proof is a modification of the argument in \cite[Theorem 3.5]{amb-kir00}, and uses curve fragments and Alberti representations. A {\em curve fragment} in $X$ is a bi-Lipschitz map $\gamma:{\rm dom}(\gamma)\to X$ where ${\rm dom}(\gamma)\subset \R$ is compact, and the set ${\rm Fr}(X)$ of curve fragments in $X$ is equipped with the topology arising from the Hausdorff metric on their graphs, see \cite[Definition 2.1]{Ba}. An {\em Alberti representation} $\mathcal A=\{\nu_\gamma,\mathbb P\}$ of a (Radon) measure $\nu$ on $X$ consists of a finite positive measure $\mathbb P$ on ${\rm Fr}(X)$ and a family $\{\nu_\gamma\}$ of probability measures on $X$ such that
\begin{itemize}
	\item[(a)] $\nu_\gamma\ll\mathcal H^1|_{{\rm Im}(\gamma)}$ $\mathbb P$-a.e. $\gamma$;
	\item[(b)] $\gamma\mapsto \nu_\gamma(B)$ is $\mathbb P$-measurable and $\displaystyle \nu(B)=\int\nu_\gamma(B)\ud \mathbb P(\gamma)$ for every Borel $B\subset X$.
\end{itemize}
Given $\varphi\in \LIP(X,\R^n)$, $z\in S^{n-1}$, $\varepsilon>0$, and a cone $C(z,\varepsilon):=\{p\in \R^n: z\cdot p\ge (1-\varepsilon)|p| \}$, we say that the Alberti representation $\mathcal A$ is in the $\varphi$-direction of $C(z,\varepsilon)$ if $(\varphi\circ\gamma)'(t)\in C(z,\varepsilon)$ a.e. $t\in {\rm dom}(\gamma)$ for $\mathbb P$-a.e. $\gamma\in {\rm Fr}(X)$. Note that if $p\in C(z,\varepsilon)$, then $|z-p/|p||<2\varepsilon$. See \cite[Section 2, Definition 5.7, and Definition 7.3 ]{Ba} for the definition of independence, $\delta$-speed and $\xi$-separation of Alberti representations used in the proof below. Alberti representations have the following very useful property.  If $\Gamma_0\subset {\rm Fr}(X)$ is $\mathbb P$-null, $E_\gamma\subset{\rm dom}(\gamma)$ is $\Leb{1}$-null for each $\gamma\notin\Gamma_0$, and $\{(\gamma,t):\gamma\notin\Gamma_0,\ t\in E_\gamma\}\subset {\rm Fr}(X)\times \R$ is $\mathbb P\times\Leb{1}$-measurable,
then for $\nu$-a.e. $x\in X$ there exists $\gamma\notin \Gamma_0$ and $t\in {\rm dom}(\gamma)\setminus E_\gamma$ with $\gamma_t=x$, cf. \cite[Proposition 2.9]{Ba}.


The following facts, which will be used in the proof pf Proposition \ref{prop:w-star-vs-metr-diff}, can be established as in the proof of Proposition \ref{prop:w-star-differential}. Let $U\subset X$ be a Borel set and $f\in \LIP(U,V^*)$. If $\mathcal A$ is an Alberti representation of $\mu|_U$, then the limit
$$ (f\circ\gamma)_t'=w^*-\lim_{{\rm dom}(\gamma)\ni t'\to t}\frac{f(\gamma_{t'})-f(\gamma_t)}{t'-t}\in V^*$$ exists for a.e. $t\in {\rm dom}(\gamma)$ for $\mathbb P$-a.e. $\gamma$. If $(U,\varphi)$ is a weak Cheeger chart, then $D_{\gamma_t}f((\varphi\circ\gamma)_t')=(f\circ\gamma)_t'$ for a.e. $t\in {\rm dom}(\gamma)$ for $\mathbb P$-a.e. $\gamma$. Finally, for $\mathbb P$-a.e. $\gamma$, we have that $$ \|D_{\gamma_t}f((\varphi\circ\gamma)'_t)\|_{V^*}=\lim_{h\to 0^+}\frac 1h\int_t^{t+h}\chi_{{\rm dom}(\gamma)}(s)\|D_{\gamma_s}f((\varphi\circ\gamma)_s')\|_{V^*}\ud s$$ a.e. $t\in {\rm dom}(\gamma)$ since $t\mapsto \chi_{{\rm dom}(\gamma)}(t)\|D_{\gamma_t}f((\varphi\circ\gamma)'_t)\|_{V^*}$ is integrable for $\mathbb P$-a.e. $\gamma$, see \cite[Theorem 1.4]{CE} and \cite[Theorem 3.5]{CJP}. 

\begin{proof}[Proof of Proposition \ref{prop:w-star-vs-metr-diff}]
	
By passing to a subset we may assume that $(U,\varphi)$ is a  $\lambda$-structured chart and $\mu|_U$ has $n$ $\xi$-separated Alberti representations with speed strictly greater than $\delta$, for some numbers $\lambda,\xi,\delta>0$. Let $z\in S^{n-1}$ and $\varepsilon>0$. By \cite[Theorem 9.5]{Ba} $\mu|_U$ has an Alberti representation in the $\varphi$-direction $C(z,\varepsilon)$ with speed greater than $\tau=\tau(n,\lambda,\xi,\delta)>0$. Thus, for  $\mu$-a.e. $x\in U$, there exists a curve fragment $\gamma:{\rm dom}(\gamma)\to X$ in the $\varphi$-direction of $C(z,\varepsilon)$ and with $\varphi$-speed at least $\tau$, and $t\in {\rm dom}(\gamma)$ with $\gamma_t=x$, $(\varphi\circ\gamma)_t'\in C(z,\varepsilon)$, and $|(\varphi\circ\gamma)'_t|\ge \tau \Lip\varphi(x)|\gamma_t'|$ such that $\md_{x}f,\ D_xf, (\varphi\circ\gamma)_t',(f\circ\gamma)_t'$ exist and satisfy $(f\circ\gamma)_t'= D_{x}f((\varphi\circ\gamma)_t')$. Moreover, we may assume that
\begin{itemize}
\item[(1)] $(f\circ\gamma)_s'= D_{\gamma_s}f((\varphi\circ\gamma)_s')$ a.e. $s\in {\rm dom}(\gamma)$;
		
\item[(2)] we have $\displaystyle \|D_{\gamma_t}f((\varphi\circ\gamma)'_t)\|_{V^*}=\lim_{h\to 0^+}\frac 1h\int_t^{t+h}\chi_{{\rm dom}(\gamma)}(s)\|D_{\gamma_s}f((\varphi\circ\gamma)_s')\|_{V^*}\ud s$,  and \newline $\displaystyle \lim_{h\to 0^+}\frac{|[t,t+h]\cap{\rm dom}(\gamma)|}{h}=1$.
\end{itemize}
Indeed, (1) and (2) can be assumed to hold by the discussion before the proof.
	
	
From the lower semicontinuity of the norm with respect to $w^*$-convergence we obtain $$\|D_{x}f((\varphi\circ\gamma)_t')\|_{V^*}=\|(f\circ\gamma)_t'\|_{V^*}\le \md_xf((\varphi\circ\gamma)_t').$$ Denoting $z_\gamma:=z-\frac{(\varphi\circ\gamma)_t'}{|(\varphi\circ\gamma)_t'|}$ we have $\|D_xf(z)\|_{V^*}\le \md_xf(z)+\|D_xf(z_\gamma)\|_{V^*}+\md_xf(z_\gamma)$. Since, for $\mu$-a.e. $x\in U$, we have that for every $\varepsilon>0$ there exist $\gamma$ and $t\in {\rm dom}(\gamma)$ as above with $|z_\gamma|<2\varepsilon$, we obtain the inequality $\|D_xf(z)\|_{V^*}\le \md_xf(z)$ for $\mu$-a.e. $x\in U$.
	
We prove the opposite inequality. Let $f_\gamma:[a,b]\to V^*$ be the extension of $f\circ\gamma:{\rm dom}(\gamma)\to V^*$ to the smallest interval $[a,b]$ containing ${\rm dom}(\gamma)$ obtained by extending linearly into the gaps. Writing $[a,b]\setminus{\rm dom}(\gamma)=\bigcup_i(a_i,b_i)$, we have that $f_\gamma'(s)=\frac{f(\gamma_{b_i})-f(\gamma_{a_i})}{b_i-a_i}$, $s\in (a_i,b_i)$, so that $\|f_\gamma'\|_{V^*}\le \LIP(f\circ\gamma)$ on $[a,b]\setminus{\rm dom}(\gamma)$ and $f_\gamma'=(f\circ\gamma)'=D_{\gamma}f((\varphi\circ\gamma)')$ a.e. on ${\rm dom}(\gamma)$. For $v\in V$ with $\|v\|_V\le 1$  and $h>0$ we have
\begin{align*}
	\Big\langle \frac{f(\gamma_{t+h})-f(\gamma_t)}{h},v\Big\rangle=\frac 1h\int_t^{t+h}&\chi_{{\rm dom}(\gamma)}(s)\langle D_{\gamma_s}f((\varphi\circ\gamma)_s'),v\rangle\ud s\\
	&+ \frac 1h\int_t^{t+h}\chi_{\R\setminus{\rm dom}(\gamma)}(s)\langle f_\gamma'(s),v\rangle\ud s.
\end{align*}
Taking supremum over $v$with $\|v\|_V\le 1$ yields the estimate
\begin{align*}
\frac{\|f(\gamma_{t+h})-f(\gamma_t)\|_{V^*}}{h}\le \frac 1h\int_t^{t+h}&\chi_{{\rm dom}(\gamma)}(s)\|D_{\gamma_s}f((\varphi\circ\gamma)_s')\|_{V^*}\ud s\\
&+\LIP(f\circ\gamma)\frac{|[t,t+h]\setminus{\rm dom}(\gamma)|}{h}.
\end{align*}
Letting $h\to 0^+$ and using (2) we obtain
\begin{align*}
	\md_xf((\varphi\circ\gamma)_t')=\lim_{h\to 0^+}\frac{\|f(\gamma_{t+h})-f(\gamma_t)\|_{V^*}}{h}\le \|D_xf((\varphi\circ\gamma)_t')\|_{V^*}.
\end{align*}
Thus $\md_xf(z)\le \|D_xf(z)\|_{V^*}+\md_x f(z_\gamma)+\|D_xf(z_\gamma)\|_{V^*}$, where $z_\gamma=z-\frac{(\varphi\circ\gamma)_t'}{|(\varphi\circ\gamma)_t'|}$ satisfies $|z_\gamma|<2\varepsilon$. Arguing as above we get $\md_xf(z)\le \|D_xf(z)\|_{V^*}$ $\mu$-a.e. $x\in U$.

	
By choosing a countable dense set $D\subset \R^n$ it follows from the argument above that $\mu$-a.e. $x\in U$ we have $\md_xf(z)=\|D_xf(z)\|_{V^*}$ for all $z\in D$. For such $x$, the equality holds for all $z\in \R^n$ by continuity and 1-homogeneity. This completes the proof.
\end{proof}

%

\end{document}